\documentclass[10pt]{article} 
\usepackage{amsmath}
\usepackage{mathtools}
\usepackage{graphicx}
\usepackage{subfigure}
\usepackage{amsthm}
\DeclarePairedDelimiter\norm{\lVert}{\rVert}
\title{%
On the perturbation of well separated matrices
}
	\author{Hariprasad M. \\
\small Department of Computational and Data Sciences\\
	Indian Institute of Science, Bangalore, India \\
mhariprasadkansur@gmail.com
}
\newtheorem{theorem}{Theorem}
\newtheorem{lemma}{Lemma}
\newtheorem{corollary}{corollary}
\begin{document}
	\maketitle
	






\begin{abstract}
A matrix is well seperated if all its Gershgorin circles are away from the unit circle and they are seperated from each other. In this article, the region of relative errors in the eigenvalues is obtained as a quadratic oval for non diagonal perturbation of well seperated matrices. Thus giving a computable relative error bound in terms of Gershgorin circle parameters. When the seperation is $\mathcal{O}(n)$ and the matrix is positive definite, an interlacing theorem for the eigenvalues under perturbation is presented. Further when the seperation is $\mathcal{O}(n^2)$, condition number of the eigenvector matrix is upper bouned to obtain the region of perturbed eigenvalue. Numerical results show the relation between diagonal entries and the magnitude of the eigenvector entries even when the matrix is not so well seperated. We exploit this trend in estimating the Perron vector using power method.

\end{abstract}


\section{Introduction}
The relative error in approximating the eigenvalue $\lambda$ of $A$ by $\tilde{\lambda}$ of $B$ is $\frac{|\lambda - \tilde{\lambda}|}{|\lambda|}$.
There are several relative error bounds \cite{eisenstat1998three} for the matrix perturbation. Like the well known Bauer-Fike \cite{chu2003perturbation} type of bound  $$\frac{|\lambda-\tilde{\lambda}|}{|\lambda|} \leq \kappa(X) \norm{A^{-1}(A-B)}.$$ 
Here $\kappa(X) = \norm{X}\norm{X^{-1}}$, is the condition number of eigenvector matrix $X$ of $A$ and $B$ is the perturbed matrix. $\lambda$ and $\tilde{\lambda}$ are the eigenvalues of the matrix $A$ and $B$. However these bounds are difficult to compute, since the eigenvector matrix is not known a priori. In this article,  first we use the knowledge of Gershgorin circles  to obtain the bounds. Then with boundon the separation we present an eigenvalue interlacing theorem for positive definite matrix with the eigenvalues of its perturbation. Using similar argument we obtain condition number estimate for a general well seperated matrix. Further we present the application of these results in numerical methods. \\

\section{On the relative errors :}

For an $n \times n $ square matrix $A$, the circle in the complex plane with center at $A(i,i)$ and radius $\sum_{i = 1}^n |A(i,j)|$ is called the $i$ th Gershgorin circle \cite{weisstein2003gershgorin}. The radius of this circle can be chosen as the minimum of  $\sum_{i = 1}^n |A(i,j)|$ or  $\sum_{j = 1}^n |A(i,j)|$. 
Gershgorin circle theorem : 
When all the Gershgorin circles of a matrix are disjoint, then each circle contains exactly one eigenvalue. When two or more circles overlap, then the union of the region contain that many eigenvalues.\\
Thus Gershgorin circle theorem gives an easily computable bounding region for the eigenvalue in the complex plane. Here analogously using the theorem we try to obtain the region of relative error in perturbing the matrix, but retaining the same diagonal entries.
Consider the matrices $A$ and $B$ which have same diagonal elements and the Gershgorin circles formed by these diagonal entries are well seperated and they lie away from the unit circle in the complex plane.  \\

The expression for relative error involves $\frac{1}{\lambda}$, and it can be related to inversion mapping in geometry (Inversive geometry \cite{morley2013inversive}). \\
Inversion with respect to a circle of radius $R$ having center at origin: The point $p$ in the Eucledian plane is mapped to the point $q$ on the ray joining origin to $p$ such that : the diatance between origin to $p$ ($r_p$) and distance between the origin to $q$ ($r_q$) satisfy the relation ,
$$R^2 = r_qr_p.$$    
In this article the inversion mapping is always with respect to the unit circle ($R=1$).  \\

A circle is mapped into a circle under the inversion ( Chapter 8 on transformations in \cite{kay1969college}). The center of the inverted circle, original circle and the origin lie on a straight line. \\

The complex map $z \to \frac{1}{z}$ is called reciprocation in geometry and it is the reflection with respect to real axis after doing the circle inversion.

Let $ae^{i \alpha}+r_1e^{i \theta}$ be a Gershgorin circle for matrix $A$, where $ae^{i \alpha}$ correspond to the diagonal entry of the matrix and $r_1$ is the absolute sum of non diagonal (row or column) entries. Let this circle contain the eigenvalue $\lambda$.
Then the circle corresponding to $\frac{1}{\lambda}$ is inside the unit circle and has the center 
\begin{align}
\frac{1}{2}\left( \frac{1}{a+r_1} + \frac{1}{a-r_1}\right)e^{-i \alpha} = \frac{a}{(a+r_1)(a-r_1)} e^{-i\alpha}.
\end{align} 
The radius of the circle is given as 
\begin{align}
\frac{1}{2}\left(\frac{1}{a-r_1} - \frac{1}{a+r_1} \right) = \frac{r_1}{(a+r_1)(a-r_1)}.
\end{align}
This can be also seen from the inverse circle mapping.
However the inverted circle in this case is also reflected about the real axis due to the complex conjugation.

Suppose the matrix $A$ is approximated by the matrix $B$ by truncating the diagonals of $A$, the $\tilde{\lambda}$ corresponding to the eigenvalue $\lambda$ lies inside the circle $ae^{i \alpha} + r_2e^{i \eta}$. 
The relative error in this approximation is given by
\begin{align}
\frac{|\lambda - \tilde{\lambda}|}{|\lambda|} = |1-z|.
\end{align}

\begin{theorem}
	The relative error in approximating eigenvalue $\lambda$ with Gershgorin circle of center $a$ and radius $r_1$ (with $a >> r_1>0 $) by a perturbed matrix eigenvalue $\tilde{\lambda}$ with same center and radius $r_2$ is given by a quadratic oval, giving the bound
	$\frac{|\lambda - \tilde{\lambda}|}{|\lambda|} \leq \frac{r_1 + r_2}{a-r_1}$. 
\end{theorem}
\begin{proof}
	The ratio $\frac{\tilde{\lambda}}{\lambda}$ is inside the region given by product of two circles.
	Which is,
	\begin{align}
	\left(ae^{i \alpha}+r_2e^{i \eta}\right)\left( \frac{a}{(a+r_1)(a-r_1)} e^{-i\alpha} + \frac{r_1}{(a+r_1)(a-r_1)}e^{i \theta} \right).
	\end{align}
	The product of two circles is in general known as cartesian oval \cite{farouki2001minkowski}, which is a curve of genus 1. 
	On furthure simplifications (note that $\theta$ and $\eta$ are arbitrary and independent),
	\begin{align}
	\left(a+r_2e^{i \eta}\right)\left( \frac{a+r_1e^{i \theta}}{(a+r_1)(a-r_1)} \right).
	\end{align}
	This gives the region of $z$, thus we have,
	\begin{align}
	z-1 = \frac{r_1^2}{(a+r_1)(a-r_1)} + \frac{a(r_1e^{i \theta}+r_2e^{i \eta}) + r_1r_2e^{i (\theta+\eta)}}{(a+r_1)(a-r_1)}.
	\end{align}
	Since $a,r_1,r_2 > 0$, we have
	\begin{align}
	|z-1| \leq   \frac{r_1^2+a(r_1+r_2) + r_1r_2}{(a+r_1)(a-r_1)} = \frac{r_1+r_2}{a-r_1}.
	\end{align}
	This serves as the upper bound for relative error.
\end{proof} 
Approximating the region by circle: 
By looking at the maximum and minimum values of the product of two circles, we can approximate the quadratic oval by the circle with center
\begin{align}
\frac{1}{2}\left( \frac{a+r_2}{a-r_1}+\frac{a-r_2}{a+r_1} \right) = \frac{a^2 +r_1r_2}{(a+r_1)(a-r_1)},
\end{align}
and radius,
\begin{align}
\frac{1}{2}\left( \frac{a+r_2}{a-r_1}-\frac{a-r_2}{a+r_1} \right) = \frac{a(r_1+r_2)}{(a+r_1)(a-r_1)}.
\end{align}

So $z-1$ is inside the circle of same radius, but center is shifted to,
\begin{align}
\frac{a^2 +r_1r_2}{(a+r_1)(a-r_1)}-1 = \frac{r_1^2 + r_1r_2}{(a+r_1)(a-r_1)}. 
\end{align}

\section{An interlacing theorem for $\mathcal{O}(n)$ seperated positive definite matrices:}
Let the matrix $A$ of dimension $n$ is positive definite and well seperated in the diagonals. The diagonal entries satisfy $|a_i-a_j| = \mathcal{O}(n)$ for $i \neq j$ and $r_i = \mathcal{O}(1)$. 
Then we have the following lemma. 
\begin{lemma}\label{l1} For the positive definite well seperated matrix $A$ with seperation $\mathcal{O}(n)$,  the unit eigenvector $x$ ( $\norm{x}_2 =1$) corresponding to the eigenvalue $\lambda_j$ falling in the Gershgorin circle with center $a_j$ has entries $|x_i| = \mathcal{O}\left( \frac{1}{n}\right)$ for $j \neq i$.
\end{lemma}
\begin{proof}
We can chose the eigenvector $x$ such that its $i^{\text{th}}$ entry $x_i >0$. Then by
Gershgorin theorem,
\begin{align}\label{gershentry}
a_ix_i +r_i > \lambda_jx_i > a_ix_i -r_i. 
\end{align}

Now we have two cases, $a_i < \lambda_j$ or $a_i > \lambda_j$. 
The case when $a_i = \lambda_j$ can only happen when $i=j$.  But we are interested in the condition when $i \neq j$. 
In the first case when $a_i < \lambda_j$ from first part of \eqref{gershentry},
\begin{align}
r_i < (\lambda_j-a_i)x_i, \\
x_i < \frac{r_i}{\lambda_j-a_i} \label{bound1}.
\end{align}

In the second case when $a_i > \lambda_j$,

\begin{align}
r_i < (a_i-\lambda_j)x_i, \\
x_i < \frac{r_i}{a_i-\lambda_j} \label{bound2}.
\end{align}

From the equations \eqref{bound1} and \eqref{bound2} we have,

\begin{align}
x_i = \mathcal{O}\left( \frac{1}{n} \right).
\end{align}
\end{proof}

Let the matrix $A$ is perturbed by matrix $S$ having $S(i,j) = \mathcal{O}\left( \frac{1}{n}\right)$ and
$S(i,i) = \mathcal{O}(1) > 0$. If $\lambda_1 \leq \lambda_2 \leq \cdots \leq \lambda_n$ are the eigenvalues of $A$ and  $\tilde{\lambda_1} \leq \tilde{\lambda_2} \leq \cdots \leq \tilde{\lambda_n}$ are eigenvalues of $A+tS$ for $t >0$ and $t = \mathcal{O}(1)$. we have the following interlacing theorem, 

\begin{theorem}
 The eigenvalues of $A$ and $A+tS$ interlace for $t > 0 $ and $t = \mathcal{O}(1)$. (i.e. $\lambda_1 \leq \tilde{\lambda_1} \leq \lambda_2 \leq \tilde{\lambda_2} \cdots \leq \lambda_n \leq \tilde{\lambda_n}$)
\end{theorem} 

\begin{proof}
Let $x$ and $y$ be the eigenvectors correspong to the Gershgorin circle and its perturbation,
$Ax = \lambda_i x$ and $(A+tS)y  = \tilde{\lambda_i} y$. 
Consider,

\begin{align}
x^T ((A+tS)-A)y = t x^TSy \\
(\tilde{\lambda_i}-\lambda_i) x^Ty = t x^TSy \label{diff} 
\end{align} 
Now by using the lemma \ref{l1}, the sign of $x^Ty$ is sign of $x_iy_i$. 
And the sign of $x^TSy$ is also the sign of $x_iy_i$. Thus we get 
$\lambda - \tilde{\lambda} > 0$ for all such eigenvalue pairs.
Also from the well seperation we have $a_i + \mathcal{O}(1) < a_{i+1} - r_{i+1}$.
Thus we get $ \tilde{\lambda_i} < \lambda_{i+1}$.  	 
\end{proof}	

\section{On bounding the condition number of eigenvector matrix}
Let the matrix $A$ of dimension $n$ and diagonal entries satisfy $|a_i-a_j| = \mathcal{O}(n^2)$ for $i \neq j$ and $r_i = \mathcal{O}(1)$. Note that here we dont have any condition on symmetry of definiteness. 
Then we have the following lemma. 
\begin{lemma}\label{l2} For the well seperated matrix $A$ with seperation $\mathcal{O}(n^2)$,  the unit eigenvector $x$ ( $\norm{x}_2 =1$) corresponding to the eigenvalue $\lambda$ falling in the Gershgorin circle with center $a_i$ has entries $|x_j| = \mathcal{O}\left( \frac{1}{n^2}\right)$ for $j \neq i$.
\end{lemma}
\begin{proof}
From the eigenvector relation $Ax = \lambda x$,

\begin{align}
a_ix_i + \sum_{k \neq i} {A(i,k)}x_k &= \lambda_j x_i \\
 |\lambda_j-a_i||x_i| &\leq r_i.  
\end{align}
This gives $|x_i| = \mathcal{O}\left( \frac{1}{n^2}\right)$ for $i \neq j$. 
\end{proof}	
Let us consider $|x_i| < \frac{k}{n^2}$, for $i \neq j$, then we have $|x_j| > 1- k^2\frac{n-1}{n^4} > 1-\frac{k^2}{n^3}$. 
Now we look at the matrix $H = X^TX$ where $X$ is eigenvector matrix of $A$.
\begin{theorem}
	The condition number of eigenvector matrix $X$ of well seperated matrix $A$ satisfies $\kappa(X) \leq \frac{n^3+3kn^2+k^2}{n^3-3kn^2-3k^2}$ 
\end{theorem}
\begin{proof}
	The diagonal entry of matrix $H$ satisfy 
	\begin{align*}
	H(i,i)  &= x^Tx \\
	H(i,i)  &= \mathcal{O}(1) + \mathcal{O}\left(\frac{1}{n^3} \right) \\
    \end{align*}
    So we have $(1-\frac{k^2}{n^3})^2 - \frac{k^2}{n^3} \leq H(i,i) \leq 1 + \frac{k^2}{n^3}$.
    Further,  $1 - \frac{3 k^2}{n^3} \leq H(i,i) \leq 1 + \frac{k^2}{n^3}$.
    The entry $H(i,j)$ for $i \neq j$ is bounded by $H(i,j) \leq \frac{2k}{n^2} + \frac{k^2}{n^3}$.
    So we have Gershgorin radius bounded by $\frac{3k}{n}$.
    So we have the condition number $\kappa(X) \leq \frac{1 + \frac{k^2}{n^3} + \frac{3k}{n}}{1 - \frac{3k^2}{n^3}-\frac{3k}{n}} = \frac{n^3+3kn^2+k^2}{n^3-3kn^2-3k^2}$.
    
\end{proof}
\begin{corollary}
When the $\mathcal{O}(n^2)$ seperated matrix $A$ is perturbed by $\Delta$, the eigenvalues $\lambda$ and $\tilde{\lambda}$ of $A$ and $A+\Delta$ satisfy
$|\lambda - \tilde{\lambda}| \leq \frac{n^3+3kn^2+k^2}{n^3-3kn^2-3k^2} \norm{\Delta}$.  
\end{corollary}	
The proof of the corollary follows from Bauer-Fike theorem.	
\section{Numerical result}
\subsection{On the relative error bound :} 
The example shown in Figure 1 is a matrix with strictly positive entries on the diagonal, random Gaussian symmetric entries in non diagonal. It is taken care that the circles are disjoint and away from the origin. It can be noted that, in Figure \ref{tc}  the center of the region overlaps with the relative error. In Figure \ref{btc} a upper Hessenberg matrix with positive entries is taken with disjoint Gershgorin circles. The center of the regions themselves form an upperbound to the relative error in this case. 
\subsection{On the magnitude of eigenvector entries :}
To test with the matrices arising in practial scenario, matrices "bfw62a" and "pde225" are considered from matrix market \cite{mm}. The matrix "bfw62a" arises in bounded finite dielectric waveguide problem, where as "pde225" arises in  discretization of a partial differential equation. Both these matrices have large and small diagonal entries but are not well seperated. The magnitude of eigenvector entries are plotted and large eigenvector entries shown to follow the trend of $\frac{1}{|A(i,i)-\lambda|}$ (as in Lemma \ref{l2}) in Figure \ref{vwg} and Figure \ref{vpde}. The relative error bound holds for these matrices and error seems to follow the trend of the bound in Figure \ref{ewg} and Figure \ref{epde}.     
\subsection{Estimating Perron vector using power method :}
From Perron Frobenius theorem \cite{pf}, for a matrix with positive entries there exists an eigenvector with all positive entries. The observed trend of eigenvector magnitude is used for estimating the perron vector of a matrix $A$ with positive entries. The vector $x$ with entry $x_i = 1/|(A(i,i)-K)|$ with some large value $K$ is given as a starting vector for the power method. For symmetric matrix with diagonals chosen randomly from one to $n$, this starting vector saves on an average three iteration for a given error. Fig \ref{perr} shows the fall in the logarithmic error with a random starting vector and vector $x$ chosen as mentioned above.    
\begin{figure}
	\subfigure [relative error and bounds for symmetric matrix of dimension 100 with half the Gerschgorian radius\label{tc}] {\includegraphics[width = 6 cm]{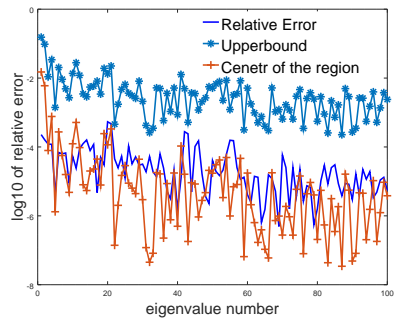}}
		\hspace{0.5cm}
	\subfigure [Logarithmic relative error, bound and the center of the region for  upper Hessenberg matrix of dimension 100 with half the Gerschgorian radius\label{btc}]{	\includegraphics[width =6 cm]{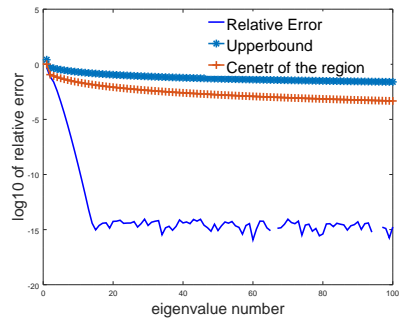}}
	
	\subfigure[Relative error and bounds for matrix arising in waveguide problem \label{ewg}]{	\includegraphics[width = 6 cm]{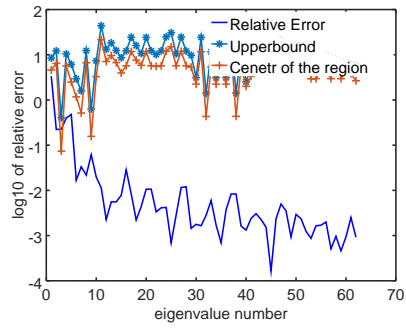}}
	\hspace{0.5cm}
	\subfigure [Relative error and bounds for matrix arising in partial differential equation \label{epde}]{\includegraphics[width = 6 cm]{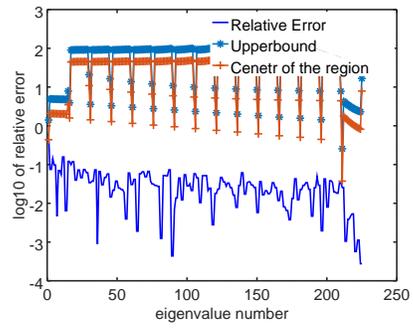}} 
	\caption{Relative error bounds}
\end{figure}

\begin{figure}
		\subfigure [absolute values of eigenvector entries corresponding to largest eigenvalue for matrix bfw62a \label{vwg}]{\includegraphics[width = 6 cm]{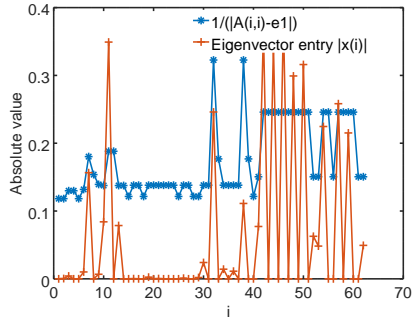}}
        \hspace{0.5cm} 
         \subfigure [absolute values of eigenvector entries corresponding to largest eigenvalue for matrix pde225 \label{vpde}]{\includegraphics[width = 6 cm]{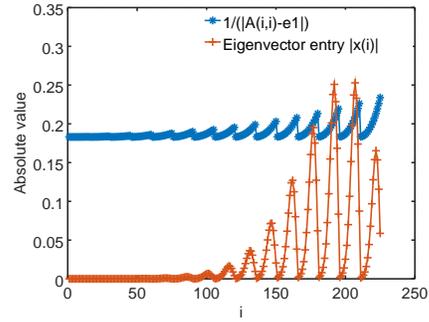}}
         \caption{Absolute value of eigenvector entries}
\end{figure}
\begin{figure}
	\centering
	\includegraphics[width = 6 cm]{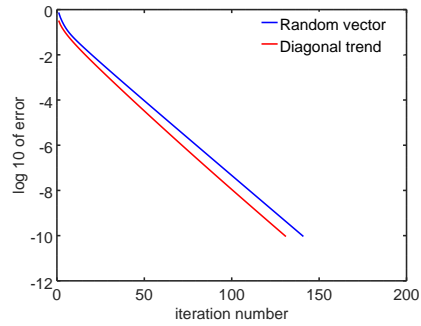}
	\caption{Error in estimating perron vector using power method for a $100$ dimensional matrix with linear diagonal entries using random starting vector and the vector based on diagonal entries.}\label{perr}
\end{figure}

\clearpage
\section{Conclusion}
 Upper bound for the relative error for non diagonal perturbations are given in terms of Gershgorin circle parameters, Numerical experiments show the goodness in the approximation of the region of relative error (quadratic oval) by a circle. The well seperated matrices have well conditioned eigenvector matrices when the seperation is $\mathcal{O}(n^2)$. The eigenvalues of positive definite $\mathcal{O}(n)$ well seperated matrices interlace with eigenvalues of perturbed matrices under small perturbations. Numerical results show that even when the matrix is not so well seperated, magnitude of eigenvector entries are related to diagonal entries. Further this diagonal trend can be exploited in better convergence of power method in estimating Perron vector.



\begin{thebibliography}{8}
\bibitem{farouki2001minkowski} Farouki, Rida T and Moon, Hwan Pyo and Ravani, Bahram, Minkowski geometric algebra of complex sets, 
\textit{Geom. Dedicata}, {\bf 85} (2001), 283--315.

\bibitem{eisenstat1998three}Eisenstat, Stanley C and Ipsen, Ilse CF, Three absolute perturbation bounds for matrix eigenvalues imply relative bounds, \textit{SIAM J. Matrix Anal. Appl.}, {\bf 20} (1998), 149--158.

\bibitem{chu2003perturbation}Chu, Eric King-wah, Perturbation of Eigenvalues for Matrix Polynomials via The Bauer--Fike Theorems, \textit{SIAM J. Matrix Anal. Appl.}, {\bf 25} (2003), 551--573.

\bibitem{weisstein2003gershgorin}Weisstein, Eric W, Gershgorin circle theorem, \textit{Wolfram Research, Inc.},(2003).

\bibitem{morley2013inversive} Morley, Frank and Morley, Frank Vigor, \textit{Inversive geometry}, Courier Corporation, 2013.

\bibitem{kay1969college}Kay, David C, \textit{College geometry}, Holt, Rinehart and Winston, 1969.


\bibitem{mm}Boisvert, Ronald F., et al. \textit{"Matrix market: a web resource for test matrix collections."} Quality of Numerical Software. Springer, Boston, MA, 1997. 125-137.

\bibitem{pf}Pillai, S. Unnikrishna, Torsten Suel, and Seunghun Cha. \textit{"The Perron-Frobenius theorem: some of its applications."} IEEE Signal Processing Magazine 22.2 (2005): 62-75. 
\end{thebibliography}
\end{document}